\PassOptionsToPackage{dvipsnames,table}{xcolor}
\documentclass{amsart}

\newcommand{\HH}{\mathbb{H}}

\newcommand{\NN}{\mathbb{N}}

\newcommand{\RR}{\mathbb{R}}

\usepackage{bm}

\newtheorem{theorem}{Theorem}[section]
\newtheorem{thm}{Theorem}
\newtheorem{lemma}[theorem]{Lemma}

\newtheorem{corollary}[theorem]{Corollary}
\newtheorem{proposition}[theorem]{Proposition}

\theoremstyle{definition}
\newtheorem{definition}[theorem]{Definition}
\newtheorem{example}[theorem]{Example}

\newtheorem{innercustomthm}{Theorem}
\newenvironment{mythm}[1]
  {\renewcommand\theinnercustomthm{#1}\innercustomthm}
  {\endinnercustomthm}

\newcommand{\flip}{\tikz[baseline=-0.76ex]{ \node (A) {$A$}; 
\node[right=1cm of A] (B) {$B$};
\path[->] (A) edge node[ fill=white, anchor=center, pos=0.5] {\footnotesize{$ij$}} (B);
} }

\newcommand{\flipp}{\tikz[baseline=-0.76ex]{ \node (A) {$A$}; 
\node[right=1cm of A] (B) {$B$};
\path[->] (A) edge node[ fill=white, anchor=center, pos=0.5] {\footnotesize{$25$}} (B);
} }

\newcommand{\flippp}{\tikz[baseline=-0.65ex]{ \node (A) {$A$}; 
\node[right=1.5cm of A] (A_1) {$A_1$};
\node[right=1.5cm of A_1] (A_2) {$A_2$};
\node[right=1.5cm of A_2] (D) {$\cdots$};
\node[right=2.25cm of D] (A_k) {$A_k$};
\node[right=1.5cm of A_k] (B) {$B$};
\path[->] (A) edge node[ fill=white, anchor=center, pos=0.5] {\footnotesize{$i_1 j_1$}} (A_1);
\path[->] (A_1) edge node[ fill=white, anchor=center, pos=0.5] {\footnotesize{$i_2 j_2$}} (A_2);
\path[->] (A_2) edge node[ fill=white, anchor=center, pos=0.5] {\footnotesize{$i_3 j_3$}} (D);
\path[->] (D) edge node[ fill=white, anchor=center, pos=0.5] {\footnotesize{$i_{k-1} j_{k-1}$}} (A_k);
\path[->] (A_k) edge node[ fill=white, anchor=center, pos=0.5] {\footnotesize{$i_k j_k$}} (B);
} }

\newcommand{\flipppp}{\tikz[baseline=-0.76ex]{ \node (A) {$A$}; 
\node[right=1cm of A] (A') {$A'$};
\path[->] (A) edge node[ fill=white, anchor=center, pos=0.5] {\footnotesize{$i j$}} (A');
} }

\newcommand{\flipppsp}{\tikz[baseline=-0.76ex]{ \node (A) {$A$}; 
\node[right=1.65cm of A] (B') {$B'$};
\path[->] (A) edge node[ fill=white, anchor=center, pos=0.5] {\footnotesize{$i B(H_3)$}} (B');
} }

\newcommand{\flipss}{\tikz[baseline=-0.76ex]{ \node (A) {$A$}; 
\node[right=1.5cm of A] (A') {$A'$};
\path[->] (A) edge node[ fill=white, anchor=center, pos=0.5] {\footnotesize{$i B(H)$}} (A');
} }

\usepackage{todonotes}

\usepackage[margin=3.35cm]{geometry}
\usepackage{amsmath,amssymb}
\usepackage{soul}
\usepackage{ulem}
\usepackage{tikz}
\usepackage{tikz-cd}
\usepackage{tkz-graph}
\usetikzlibrary{external,backgrounds}
\usetikzlibrary{arrows}
\usetikzlibrary{tikzmark}
\usepackage{enumerate}
\usepackage{graphicx}
\usepackage{caption}
\usetikzlibrary {graphs}
\usepackage{graphicx,subcaption}
\usetikzlibrary{decorations.markings}
\usepackage{dsfont}
\usetikzlibrary{shapes.geometric}
\usepackage{alphalph}
\usepackage[all]{xy}
\usepackage{multicol}
\usetikzlibrary{shapes.geometric}
\usepackage{float}
\usetikzlibrary{decorations.markings}
\usepackage{hyperref}
\hypersetup{colorlinks,linkcolor=blue,colorlinks=true,allcolors=blue,linktocpage}
\usepackage[dvipsnames]{xcolor}
\newcommand{\defn}[1]{\emph{\textcolor{gray}{#1}}}
\usetikzlibrary{intersections,decorations.markings}
\usepackage{pgfplots}
\pgfplotsset{compat=newest}
    \usetikzlibrary{intersections,pgfplots.fillbetween}
    \usepgfplotslibrary{fillbetween}
    \pgfdeclarelayer{pre main}
\usetikzlibrary {arrows.meta,bending}
\usepackage{algorithm}
\usepackage{algpseudocode}
\algrenewcommand\algorithmicrequire{\textbf{Input:}}
\algrenewcommand\algorithmicensure{\textbf{Output:}}
\usepackage{mathrsfs} 
\usetikzlibrary{backgrounds}
\usepackage{adjustbox}
\usepackage{ragged2e}
\usepackage{caption}

\title[Source characterization of the hypergraphic posets]{Source characterization of the hypergraphic posets}

\author{Félix Gélinas}
\address{Department of Mathematics and Statistics, York University, Toronto}
\email{felixgel@yorku.ca}
\urladdr{https://felixgelinas.github.io/}

\begin{document}

\begin{abstract}

    For a hypergraph $\HH$ on $[n]$, the hypergraphic poset $P_\HH$ is the transitive closure of the oriented $1$-skeleton of the hypergraphic polytope $\Delta_\HH$. In a recent paper, N. Bergeron and V. Pilaud provided a characterization of $P_\HH$ based on the sources of acyclic orientations for interval hypergraphs. The goal of this work is to extend this source characterization of $P_\HH$ for arbitrary hypergraphs on $[n]$.
\end{abstract}

\maketitle
\normalem 

\tableofcontents

\section{Introduction}

Fix an integer $n \geq 1$ and let $(e_i)_{i \in [n]}$ be the standard basis of $\mathbb{R}^n$. M. Aguiar and F. Ardila~\cite{aguiar2017hopfmonoidsgeneralizedpermutahedra} introduced the \defn{hypergraphic polytope} $\Delta_\HH$ associated with a hypergraph $\HH$ on the vertex set $[n]$ as the Minkowski sum $\Delta_\HH := \sum_{H \in \HH} \Delta_H$, where $\Delta_H$ is the simplex formed by the convex hull of the points $\{e_h \mid h \in H\}$ in $\mathbb{R}^n$. The singleton hyperedges in $\HH$ do not influence the combinatorial structure of $\Delta_\HH$, as they only contribute translational effects. Hence, we assume that $\{i\} \in \HH$ for every $i \in [n]$.

N. Bergeron and V. Pilaud~\cite{bergeron2024intervalhypergraphiclattices} defined the \defn{hypergraphic poset} $P_\HH$ as the transitive closure of the $1$-skeleton of $\Delta_\HH$ oriented along the direction $\omega := (n-1, n-3, \dots, 3-n, 1-n)$.

A natural question in lattice theory is to characterize the hypergraph $\HH$ for which $P_\HH$ is a lattice. This question has been answered for all hypergraphic posets coming from graphical zonotopes in~\cite{pilaud2024acyclic} and for interval hypergraphic posets~\cite{bergeron2024intervalhypergraphiclattices}. It was also studied in~\cite{barnard2021lattices} for graph associahedra~\cite{carr2006coxeter}.

An \defn{orientation} of a hypergraph $\HH$ is a map $A$ that assigns to each hyperedge $H$ a distinguished vertex $A(H) \in H$, called the \defn{source} of $H$. An orientation is \defn{acyclic} if the directed graph on $[n]$ with edges $(h, A(H))$ for all $h \in H \setminus \{A(H)\}$ and $H \in \HH$ contains no directed cycles.

The set of acyclic orientations of $\HH$ inherits a natural partial order from the combinatorial structure of the hypergraphic polytope $\Delta_\HH$, as described in~\cite[Thm.~2.18]{Benedetti_2019}. Specifically, the vertices of $\Delta_\HH$ correspond to acyclic orientations, and the $1$-skeleton oriented under the direction $\omega$ induces a poset structure on these orientations. Hence, $A \leq B$ if there is a directed path from $A$ to $B$ in the 1-skeleton of $\Delta_\HH$ that is oriented along $\omega$.

To characterize when the interval hypergraphic poset is a lattice,~\cite{bergeron2024intervalhypergraphiclattices} gave a description of $P_\HH$ based on the sources of acyclic orientations for interval hypergraphs.

To extend this tool to all hypergraphic posets, the goal of this work is to generalize the source characterization of $P_\HH$ to arbitrary hypergraphs on $[n]$. Therefore, we prove the following result.

\begin{thm} \label{thm}
    For any acyclic orientations $A$ and $B$ of a hypergraph $\HH$,

    $$A \leq B \Longleftrightarrow A(H)\leq B(H) \textit{ for all } H \in \HH.$$
\end{thm}

\cite[Remark 3.14]{bergeron2024intervalhypergraphiclattices} gives us the forward implication for arbitrary hypergraphs. Hence, this paper will provide a proof of the backward implication of Theorem~\ref{thm}.

Extending this characterization to arbitrary hypergraphs provides a meaningful tool for studying hypergraphic posets and polytopes, as it gives a complete combinatorial interpretation of their geometric and order-theoretic structures. This includes well-known polytopes such as 

\begin{multicols}{2}
\begin{itemize}
    \item the permutahedra,

    \item the graphical zonotopes,

    \item the graph associahedra~\cite{carr2006coxeter},

    \item the nestohedra~\cite{feichtner2005matroid},
    
    \item the associahedra~\cite{shnider1993quantum,loday2004realization},
    
    \item the Pitman-Stanley polytopes~\cite{stanley2002polytope},
    
    \item the freehedra~\cite{saneblidze2009bitwisted},
    
    \item the fertilitopes~\cite{defant2023fertilitopes},
    
    \item the cyclohedra~\cite{Bott2017},

    \item the master polytopes~\cite{agnarsson2013flag},

    \item the hyper-permutahedra~\cite{agnarsson2016special},

    \item the multiplihedra~\cite{forcey2008convex,ardila2013lifted,Chapoton_2025},

    \item the constrainahedra~\cite{Tierney2016,Poliakova2021,bottman2022constrainahedra,Chapoton_2025}.

\end{itemize}
\end{multicols}

Note that hypergraphic polytopes have previously been studied in~\cite{postnikov2008faces,agnarsson2009minkowski,postnikov2009permutohedra,agnarsson2013flag,agnarsson2016special}, and hold particular interest as it forms a subfamily of the generalized permutahedra introduced in~\cite{postnikov2008faces,postnikov2009permutohedra}.

\section{Hypergraphic polytopes, orientations, flips, and hypergraphic posets}

\subsection{Hypergraphic polytope $\Delta_\HH$}

A \defn{hypergraph} $\HH$ on $[n]:= \{1, \dots, n\}$ is a set of subsets of $[n]$. By convention, we always assume that all singletons $\{i\}$ for $i \in [n]$ are included in $\HH$. Let $e_i \in \mathbb{R}^n$ denote the $i^{th}$ standard basis vector in $\mathbb{R}^n$. The \defn{hypergraphic polytope} $\Delta_\HH$ is the Minkowski sum

$$\Delta_\HH:=\sum_{H\in \HH} \Delta_H,$$

\noindent where $\Delta_H$ is the simplex given by the convex hull of the set of points $\{e_h | h\in H\}$ in $\RR^n$. 

\begin{example}

    Figure \ref{figminkow} illustrates the operation to construct the hypergraphic polytope $\Delta_\HH$ associated to the hypergraph $\HH=\{123,13,124\}$. The polytope $\Delta_\HH$ is the Minkowski sum of the simplices $\Delta_{123}$, $\Delta_{13}$ and $\Delta_{124}$.
    \vspace{-0.25cm}
    \begin{figure}[H]
	\centering
        \begin{tikzpicture}[scale=1.2]

		\node at (-6,0){\begin{tikzpicture}
				\node(1) at (1,0,0,0){\textcolor{ Blue!90}{$\bullet$}};
				\node(2) at (0,1,0,0){\textcolor{ Blue!90}{$\bullet$}};
				\node(3) at (0,0,1,0){\textcolor{ Blue!90}{$\bullet$}};
				
				\draw (1.center) -- (2.center);
				\draw (1.center) -- (3.center);
				\draw (3.center) -- (2.center);
				
				\begin{pgfonlayer}{background}
					\draw[Blue!45,fill=Blue!45,line width=0mm,line cap=round,line join=round] (1.center) -- (2.center) -- (3.center) -- cycle;
				\end{pgfonlayer}
				
				\node at (1,0,0,0){\textcolor{ Blue!90}{$\bullet$}};
				\node at (0,1,0,0){\textcolor{ Blue!90}{$\bullet$}};
				\node at (0,0,1,0){\textcolor{ Blue!90}{$\bullet$}};
				
		\end{tikzpicture}};
		
		\node at (-4.75,0){\Huge $+$};
		
		\node at (-3.5,0){\begin{tikzpicture}
				\node(1) at (1,0,0,0){\textcolor{ Yellow}{$\bullet$}};
				\node(3) at (0,0,1,0){\textcolor{ Yellow}{$\bullet$}};
				
				\draw (1.center) -- (3.center);
				
				\node at (1,0,0,0){\textcolor{ Red!95}{$\bullet$}};
				\node at (0,0,1,0){\textcolor{ Red!95}{$\bullet$}};
		\end{tikzpicture}};
		
		\node at (-2.25,0){\Huge $+$};
		
		\node at (-1,0){\begin{tikzpicture}
				\node(1) at (1,0,0,0){\textcolor{ Orange}{$\bullet$}};
				\node(2) at (0,1,0,0){\textcolor{ Orange}{$\bullet$}};
				\node(4) at (0,0,0,1){\textcolor{ Orange}{$\bullet$}};
				
				\draw (1.center) -- (2.center);
				\draw (1.center) -- (4.center);
				\draw (4.center) -- (2.center);
				
				\begin{pgfonlayer}{background}
					\draw[Orange!45,fill=Orange!45,line width=0mm,line cap=round,line join=round] (1.center) -- (2.center) -- (4.center) -- cycle;
				\end{pgfonlayer}
				
				\node at (1,0,0,0){\textcolor{ Orange}{$\bullet$}};
				\node at (0,1,0,0){\textcolor{ Orange}{$\bullet$}};
				\node at (0,0,0,1){\textcolor{ Orange}{$\bullet$}};
				
		\end{tikzpicture}};
		
		\node at (0,0){\Huge $=$};
		
		\node at (2.75,0){\begin{tikzpicture}
				\node(111) at (3,0,0,0){\textcolor{ Blue!90}{$\bullet$}};
				\node(212) at (1,2,0,0){\textcolor{ Blue!90}{$\bullet$}};
				\node(114) at (2,0,0,1){\textcolor{ Blue!90}{$\bullet$}};
				\node(214) at (1,1,0,1){\textcolor{ Blue!90}{$\bullet$}};
				\node(331) at (1,0,2,0){\textcolor{ Blue!90}{$\bullet$}};
				\node(332) at (0,1,2,0){\textcolor{ Blue!90}{$\bullet$}};
				\node(334) at (0,0,2,1){\textcolor{ Blue!90}{$\bullet$}};
				\node(234) at (0,1,1,1){\textcolor{ Blue!90}{$\bullet$}};
				\node(232) at (0,2,1,0){\textcolor{ Blue!90}{$\bullet$}};
				
				\draw (111.center) -- (212.center);
				\draw[dashed] (111.center) -- (114.center);
				\draw (111.center) -- (331.center);
				\draw[dashed] (212.center) -- (214.center);
				\draw (212.center) -- (232.center);
				\draw[dashed] (214.center) -- (114.center);
				\draw[dashed] (214.center) -- (234.center);
				\draw[dashed] (114.center) -- (334.center);
				\draw[dashed] (232.center) -- (234.center);
				\draw (232.center) -- (332.center);
				\draw (332.center) -- (334.center);
				\draw (332.center) -- (331.center);
				\draw (331.center) -- (334.center);
				\draw[dashed] (234.center) -- (334.center);
				
				\begin{pgfonlayer}{background}
					\draw[Orange!45,fill=Orange!45,line width=0mm,line cap=round,line join=round] (331.center) -- (332.center) -- (334.center) -- cycle;
				\end{pgfonlayer}
				
				\node(431) at (1,0,1,1){};
				
				\begin{pgfonlayer}{background}
					\draw[Blue!45,fill=Blue!45,line width=0mm,line cap=round,line join=round] (431.center) -- (114.center) -- (214.center) -- cycle;
				\end{pgfonlayer}
				
				\node at (3,0,0,0){\textcolor{ Black}{$\bullet$}};
				\node at (1,2,0,0){\textcolor{ Red!95}{$\bullet$}};
				\node at (2,0,0,1){\textcolor{ Blue!90}{$\bullet$}};
				\node at (1,1,0,1){\textcolor{ Blue!90}{$\bullet$}};
				\node at (1,0,2,0){\textcolor{ Orange}{$\bullet$}};
				\node at (0,1,2,0){\textcolor{ Orange}{$\bullet$}};
				\node at (0,0,2,1){\textcolor{ Orange}{$\bullet$}};
				\node at (0,1,1,1){\textcolor{ Black}{$\bullet$}};
				\node at (0,2,1,0){\textcolor{ Red!95}{$\bullet$}};
				
		\end{tikzpicture}};
	
		\node at (-6,-1){$\Delta_{123}$};
		\node at (-3.45,-1){$\Delta_{13}$};
		\node at (-1,-1){$\Delta_{124}$};

	\end{tikzpicture}
        \caption{Hypergraphic polytope $\Delta_\HH$ for $\HH=\{123,13,124\}$.}
        \label{figminkow}
\end{figure}
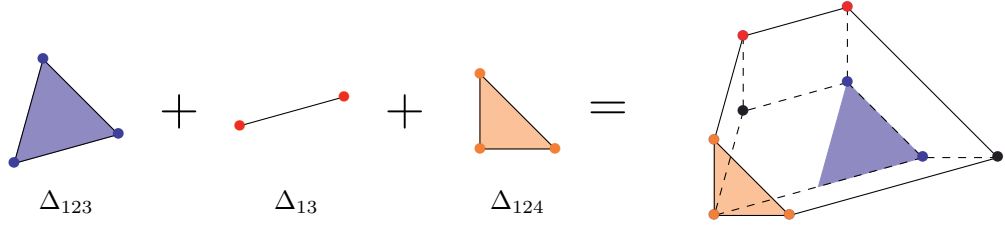
\end{example}

\vspace{-0.25cm}

\subsection{Orientations, increasing flips, and hypergraphic poset $P_\HH$}

In this section, we recall some notions and results given by~\cite{Benedetti_2019} and~\cite{bergeron2024intervalhypergraphiclattices} that are essential to introduce a combinatorial definition of the central object: the hypergraphic poset $P_\HH$. Let us first introduce some key notions.

\begin{definition}\label{orientation}
    An orientation of $\HH$ is a map $A$ from $\HH$ to $[n]$ such that $A(H) \in H$ for all $H \in \HH$. The orientation $A$ is \defn{acyclic} if there is no $H_1, \dots , H_k$ with $k \geq 2$ such that $A(H_{i+1}) \in H_i \setminus \{A(H_i)\}$ for $i \in [k - 1]$ and $A(H_1) \in H_k \setminus \{A(H_k)\}$.
\end{definition}

Let $A$ be an orientation of a hypergraph $\HH$. We refer to the \defn{source} of a hyperedge $H\in \HH$ of an orientation $A$ as the image $A(H)$. Throughout this paper we will denote an orientation $A$ of a hypergraph $\HH$ by the source sequence $(A(H))_{H\in \HH}$.

\begin{definition}\label{flip}
    Two orientations $ A\neq B$ of $\HH$ are related by an \defn{increasing flip} if there exist $1 \leq i < j \leq n$ such that for all $H\in \HH$,

    \begin{itemize}
        \item  if $A(H)\neq B(H)$ then $A(H)=i$ and $B(H)=j$, and
        \item  if $\{i,j\}\subseteq H$, then $A(H)=i \Longleftrightarrow B(H)=j$.
    \end{itemize}

    We denote a flip that relates $A$ and $B$ by \flip \hspace{-0.2cm}.
\end{definition}

See Figure~\ref{fig} for an example of an increasing flip.

Since $\Delta_\HH$ is a generalized permutahedra~\cite{postnikov2008faces,postnikov2009permutohedra}, this leads to a natural surjection from the faces of the permutahedron to the faces of $\Delta_\HH$. We focus on the surjection $\mathcal{O}$ described in~\cite[Lem. $2.9$]{Benedetti_2019} from the set of permutations on $[n]$ to acyclic orientations of $\HH$.

\begin{definition} \label{surj}
    For a permutation $\pi$ of $[n]$, the \defn{orientation $\mathcal{O}_\pi$} of $\HH$ is defined for all $H \in \HH$ by
    \vspace{-0.05cm}
    $$\mathcal{O}_\pi(H):=\pi(\min\{j|\pi(j)\in H\}).$$
\end{definition}

\begin{proposition}[{\cite[Lem.~2.9]{Benedetti_2019}; \cite[Prop.~2.14]{bergeron2024intervalhypergraphiclattices}}
] \label{adjacent}
    For any hypergraph $\HH$ on $[n]$,
    \begin{itemize}
        \item the map $\mathcal{O}$ is a surjection from the permutations of $[n]$ to the acyclic orientations of $\HH$,
        \item two distinct acyclic orientations $A, B$ of $\HH$ are related by the flip \flip if and only if there are permutations $\pi_A, \pi_B$ of $[n]$ such that $\pi_A=\pi_B(ij)$ and $\mathcal{O}_{\pi_A} = A$ and $\mathcal{O}_{\pi_B} = B$. 
    \end{itemize} 
    In other words, the $1$-skeleton of $\Delta_\HH$ is isomorphic to the graph obtained by contracting the fibers of $\mathcal{O}$ in the graph of the permutahedron. 
\end{proposition} 

\begin{example}
    Figure~\ref{fig121} illustrates Proposition~\ref{adjacent} for the hypergraph $\HH=\{12,34,234\}$.
    
     \input{permu}
\end{example}

\pagebreak

From this perspective,~\cite{Benedetti_2019} and~\cite{bergeron2024intervalhypergraphiclattices} give us a combinatorial correspondence between the graph of $\Delta_\HH$ and the notion of increasing flip on acyclic orientations of $\HH$.

\begin{proposition}[{\cite[Thm.~2.18]{Benedetti_2019}; \cite[Prop.~2.8]{bergeron2024intervalhypergraphiclattices}}]
    The graph of $\Delta_\HH$ oriented in the direction $\omega$ is isomorphic to the increasing flip graph on acyclic orientations of $\HH$.
\end{proposition} 

This leads to the following combinatorial definition of the hypergraphic poset $P_\HH$:

\begin{definition}
    The \defn{hypergraphic poset} $P_\HH$ is the transitive closure of the increasing flip graph on acyclic orientations of $\HH$.
\end{definition}

\section{Paths, and Flips} \label{new}

We now turn our attention to the notion of flips between two vertices $i$ and $j$ from an orientation in a hypergraph $\HH$, exploring whether such flips create a cycle or preserve the acyclicity of the hypergraph. To understand this, we characterize these flips in terms of the existing paths between the two vertices $i$ and $j$ in a given orientation of $\HH$. This allows us to study the flip graph on acyclic orientations of a hypergraph by analyzing the structure of paths within a given orientation.

\subsection{Paths in a given orientation of $\HH$}

In the following section, we define the notion of paths in a given orientation $A$ of a hypergraph $\HH$.

\begin{definition}
    Given an orientation $A$ on a hypergraph $\HH$, we define a \defn{path} $\lambda$ of length $m$ in $A$ as a sequence of $m+1$ vertices $[h_0, \dots, h_m]$ such that, for each $0 \leq i < m$, there exists a hyperedge $H \in \HH$ with $A(H) = h_i$ and $h_{i+1} \in H$.
\end{definition}

\begin{example}

    Consider the hypergraph $\HH = \{H_1=1245,H_2=23,H_3=35,H_4=46,H_5=356\}$ with the orientation $A = \mathcal{O}_{124635}$. Using this orientation, we will compute all the paths $\lambda_i$ from vertex $1$ to vertex $5$ in $A$. Since $5 \in H_1$, we observe that $\lambda_1 = [1, 5]$ is a path from the vertex $1$ to the vertex $5$ in $A$. Figure~\ref{fig3} graphically illustrates all the paths in $A$. From this figure, we observe that the other paths from vertex $1$ to vertex $5$ in $A$ are

    $$\begin{matrix}\lambda_2=[1,4,6,5] && \lambda_3=[1,4,6,3,5] && \lambda_4=[1,2,3,5] \end{matrix}.$$

    \begin{figure}[H]
    \centering

\begin{tikzpicture}[
    mid arrow1/.style={postaction={decorate,decoration={markings, mark=at position .6 with {\arrow[rotate=8]{Stealth}}
    }}},mid arrow/.style={postaction={decorate,decoration={markings, mark=at position .6 with {\arrow{Stealth}}
    }}},mid arrow2/.style={postaction={decorate,decoration={markings, mark=at position .6 with {\arrow[rotate=-8]{Stealth}}
    }}},mid arrow3/.style={postaction={decorate,decoration={markings, mark=at position .6 with {\arrow[rotate=7]{Stealth}}
    }}},mid arrow4/.style={postaction={decorate,decoration={markings, mark=at position .55 with {\arrow[rotate=7]{Stealth}}
    }}},scale=0.825]

\draw[ red,fill=red!30,ultra thick] plot [smooth] coordinates {(0,0)(1,0.25)(2,0)(1.75,1)(2,2)(1,1.75)(0,2)(0.25,1)(0,0)};
\draw[ blue,ultra thick] plot [smooth] coordinates {(2,0)(3,-0.25)(4,0)};
\draw[ orange,ultra thick] plot [smooth] coordinates {(0,2)(-0.25,3)(0,4)};
\draw[ violet,fill=violet!30,ultra thick] plot [smooth] coordinates {(0,4)(3.25,3.25)(4,0)(3.25,1.25)(2,2)(1.25,3.25)(0,4)};
\draw[ultra thick] plot [smooth] coordinates {(4,0)(2.5,0.75)(2,2)};

\node[draw=yellow, fill=yellow, ellipse, minimum width=0.5cm, minimum height=0.5cm, inner sep=5pt] at (0,0) {};
\node[draw=yellow, fill=yellow, ellipse, minimum width=0.5cm, minimum height=0.5cm, inner sep=5pt] at (2,0) {};
\node[draw=yellow, fill=yellow, ellipse, minimum width=0.5cm, minimum height=0.5cm, inner sep=5pt] at (4,0) {};
\node[draw=yellow, fill=yellow, ellipse, minimum width=0.5cm, minimum height=0.5cm, inner sep=5pt] at (0,2) {};
\node[draw=yellow, fill=yellow, ellipse, minimum width=0.5cm, minimum height=0.5cm, inner sep=5pt] at (2,2) {};
\node[draw=yellow, fill=yellow, ellipse, minimum width=0.5cm, minimum height=0.5cm, inner sep=5pt] at (0,4) {};

\node at (0,0) {$1$};
\node at (2,0) {$2$};
\node at (4,0) {$3$};
\node at (0,2) {$4$};
\node at (2,2) {$5$};
\node at (0,4) {$6$};

\draw[-stealth, thick] (5,2)--(7,2);
\node at (6,2.25) {$\HH$ oriented by $A$};

\draw[ mid arrow1, red,ultra thick] plot [smooth] coordinates {(8,0)(9,0.25)(10,0)};
\draw[ mid arrow, red,ultra thick] plot [smooth] coordinates {(8,0)(9,1)(10,2)};
\draw[ mid arrow2, red,ultra thick] plot [smooth] coordinates {(8,0)(8.25,1)(8,2)};
\draw[ mid arrow1, orange,ultra thick] plot [smooth] coordinates {(8,2)(7.75,3)(8,4)};
\draw[ mid arrow3, violet,ultra thick] plot [smooth] coordinates {(8,4)(9.25,3.25)(10,2)};
\draw[ mid arrow4, violet,ultra thick] plot [smooth] coordinates {(8,4)(11.25,3.25)(12,0)};
\draw[ mid arrow2, blue,ultra thick] plot [smooth] coordinates {(10,0)(11,-0.25)(12,0)};
\draw[ mid arrow2, black,ultra thick] plot [smooth] coordinates {(12,0)(11.25,1.25)(10,2)};

\node[draw=yellow, fill=yellow, ellipse, minimum width=0.5cm, minimum height=0.5cm, inner sep=5pt] at (8,0) {};
\node[draw=yellow, fill=yellow, ellipse, minimum width=0.5cm, minimum height=0.5cm, inner sep=5pt] at (10,0) {};
\node[draw=yellow, fill=yellow, ellipse, minimum width=0.5cm, minimum height=0.5cm, inner sep=5pt] at (12,0) {};
\node[draw=yellow, fill=yellow, ellipse, minimum width=0.5cm, minimum height=0.5cm, inner sep=5pt] at (8,2) {};
\node[draw=yellow, fill=yellow, ellipse, minimum width=0.5cm, minimum height=0.5cm, inner sep=5pt] at (10,2) {};
\node[draw=yellow, fill=yellow, ellipse, minimum width=0.5cm, minimum height=0.5cm, inner sep=5pt] at (8,4) {};

\node at (8,0) {$1$};
\node at (10,0) {$2$};
\node at (12,0) {$3$};
\node at (8,2) {$4$};
\node at (10,2) {$5$};
\node at (8,4) {$6$};

\end{tikzpicture}

    \caption{Graphical representation of $\HH=\{\textcolor{red}{1245},\textcolor{blue}{23},35,\textcolor{orange}{46},\textcolor{violet}{356}\}$ oriented by $A=\mathcal{O}_{124635}$. }
    \label{fig3}
\end{figure}
\end{example}

\begin{definition}
    Given an orientation $A$ on a hypergraph $\HH$ and a path $\lambda$, we define a \defn{subpath} $\alpha$ of $\lambda$ in $A$ as a sequence of consecutive vertices in $\lambda$. Formally, if $\lambda = [h_0, h_1, \ldots, h_m]$, then a subpath $\alpha$ of $\lambda$ can be expressed as $[h_i, h_{i+1}, \ldots, h_j]$ for some $0 \leq i < j \leq m$. We denote this relationship by $\alpha \subseteq \lambda$.
\end{definition}

\begin{definition}
    Given an orientation $A$ on a hypergraph $\HH$, a path of length $1$ in $A$ will be referred to as an \defn{edge}. We say that a path $\lambda$ in $A$ is a \defn{non-edge path} if the length of $\lambda$ is strictly greater than $1$.
\end{definition}

\subsection{Pre-coherent flips, and coherent flips}

In the following section, we define different notions of flips.

\begin{definition} \label{preco}
    Let $A$ and $B$ be two distinct orientations of $\HH$ related by a flip \flip and where $A$ is acyclic. This flip is a \defn{pre-coherent flip} if the only paths from $i$ to $j$ in $A$ are edges.
\end{definition}

\begin{example}

    Figure~\ref{fig} shows an example of two orientations $A$ and $B$, on the hypergraph $\HH=\{12,235,34,45\}$ related by the flip \flipp and graphically demonstrates that this flip is not pre-coherent. In this example we have that $i=2$ and $j=5$. This flip is not pre-coherent since we have that $[2,3,4,5]$ is a path from vertex $2$ to vertex $5$ that has length 3.
    \begin{figure}[H]
    \centering
   \begin{tikzpicture}[
    mid arrow/.style={postaction={decorate,decoration={markings, mark=at position .6 with {\arrow{Stealth}}
    }}},]


    \draw[mid arrow] ({1.5*cos(72)},{1.5*sin(72)}) -- ({1.5*cos(0)},{1.5*sin(0)});
    \draw[mid arrow] ({1.5*cos(288)},{1.5*sin(288)}) -- ({1.5*cos(216)},{1.5*sin(216)});
    \draw[mid arrow] ({1.5*cos(216)},{1.5*sin(216)}) -- ({1.5*cos(144)},{1.5*sin(144)});
    \draw[mid arrow] ({1.5*cos(144)},{1.5*sin(144)}) -- ({1.5*cos(72)},{1.5*sin(72)});
    \draw[mid arrow] ({1.5*cos(216)},{1.5*sin(216)}) -- ({1.5*cos(0)},{1.5*sin(0)});
    
    \node[draw=yellow, fill=yellow, ellipse, minimum width=0.5cm, minimum height=0.5cm, inner sep=5pt] at ({1.5*cos(0)},{1.5*sin(0)}) {};
    \node[draw=yellow, fill=yellow, ellipse, minimum width=0.5cm, minimum height=0.5cm, inner sep=5pt] at ({1.5*cos(72)},{1.5*sin(72)}) {};
    \node[draw=yellow, fill=yellow, ellipse, minimum width=0.5cm, minimum height=0.5cm, inner sep=5pt] at ({1.5*cos(144)},{1.5*sin(144)}) {};
    \node[draw=yellow, fill=yellow, ellipse, minimum width=0.5cm, minimum height=0.5cm, inner sep=5pt] at ({1.5*cos(216)},{1.5*sin(216)}) {};
    \node[draw=yellow, fill=yellow, ellipse, minimum width=0.5cm, minimum height=0.5cm, inner sep=5pt] at ({1.5*cos(288)},{1.5*sin(288)}) {};

    \node (5) at ({1.5*cos(0)},{1.5*sin(0)}) {$5$};
    \node  (4) at ({1.5*cos(72)},{1.5*sin(72)}) {$4$};
    \node (3) at ({1.5*cos(144)},{1.5*sin(144)}) {$3$};
    \node (2) at ({1.5*cos(216)},{1.5*sin(216)}) {$2$};
    \node (1) at ({1.5*cos(288)},{1.5*sin(288)}) {$1$};

    \draw[-stealth] (2.25,0) -- (3.75,0);


    \draw[mid arrow] ({6+1.5*cos(72)},{1.5*sin(72)}) -- ({6+1.5*cos(0)},{1.5*sin(0)});
    \draw[mid arrow] ({6+1.5*cos(288)},{1.5*sin(288)}) -- ({6+1.5*cos(216)},{1.5*sin(216)});
    \draw[mid arrow,color=blue] ({6+1.5*cos(0)},{1.5*sin(0)}) -- ({6+1.5*cos(144)},{1.5*sin(144)});
    \draw[mid arrow] ({6+1.5*cos(144)},{1.5*sin(144)}) -- ({6+1.5*cos(72)},{1.5*sin(72)});
    \draw[mid arrow,color=blue] ({6+1.5*cos(0)},{1.5*sin(0)}) -- ({6+1.5*cos(216)},{1.5*sin(216)});

    \node[draw=yellow, fill=yellow, ellipse, minimum width=0.5cm, minimum height=0.5cm, inner sep=5pt] at ({6+1.5*cos(0)},{1.5*sin(0)}) {};
    \node[draw=yellow, fill=yellow, ellipse, minimum width=0.5cm, minimum height=0.5cm, inner sep=5pt] at ({6+1.5*cos(72)},{1.5*sin(72)}) {};
    \node[draw=yellow, fill=yellow, ellipse, minimum width=0.5cm, minimum height=0.5cm, inner sep=5pt] at ({6+1.5*cos(144)},{1.5*sin(144)}) {};
    \node[draw=yellow, fill=yellow, ellipse, minimum width=0.5cm, minimum height=0.5cm, inner sep=5pt] at ({6+1.5*cos(216)},{1.5*sin(216)}) {};
    \node[draw=yellow, fill=yellow, ellipse, minimum width=0.5cm, minimum height=0.5cm, inner sep=5pt] at ({6+1.5*cos(288)},{1.5*sin(288)}) {};

    \node (55) at ({6+1.5*cos(0)},{1.5*sin(0)}) {$5$};
    \node (44) at ({6+1.5*cos(72)},{1.5*sin(72)}) {$4$};
    \node (33) at ({6+1.5*cos(144)},{1.5*sin(144)}) {$3$};
    \node (22) at ({6+1.5*cos(216)},{1.5*sin(216)}) {$2$};
    \node (11) at ({6+1.5*cos(288)},{1.5*sin(288)}) {$1$};

    \node at (0,-2.25) {$\HH$ oriented by $A$};
    \node at (6.15,-2.25) {$\HH$ oriented by $B$};
    \node at (3,0.25) {flip of};
    \node at (3,-0.25) { $2$ and $5$};

    \begin{pgfonlayer}{background}
        \draw[red!30,fill=red!30,line width=10mm,line cap=round,line join=round] (2.center) -- (3.center) -- (5.center) -- cycle;
        \draw[opacity=.5,cyan!30,fill=cyan!30,line width=10mm,line cap=round,line join=round] (1.center) -- (2.center) -- cycle;
        \draw[opacity=.5,pink!50,fill=pink!50,line width=10mm,line cap=round,line join=round] (3.center) -- (4.center) -- cycle;
        \draw[opacity=.5,violet!50,fill=violet!50,line width=10mm,line cap=round,line join=round] (4.center) -- (5.center) -- cycle;
    \end{pgfonlayer}

    \begin{pgfonlayer}{background}
        \draw[red!30,fill=red!30,line width=10mm,line cap=round,line join=round] (22.center) -- (33.center) -- (55.center) -- cycle;
        \draw[opacity=.5,cyan!30,fill=cyan!30,line width=10mm,line cap=round,line join=round] (11.center) -- (22.center) -- cycle;
        \draw[opacity=.5,pink!50,fill=pink!50,line width=10mm,line cap=round,line join=round] (33.center) -- (44.center) -- cycle;
        \draw[opacity=.5,violet!50,fill=violet!50,line width=10mm,line cap=round,line join=round] (44.center) -- (55.center) -- cycle;
    \end{pgfonlayer}

    \end{tikzpicture}
    \caption{Graphical representation of a non pre-coherent flip on $\HH=\{12,235,34,45\}$ between the orientation $A=\mathcal{O}_{1234}$ and the orientation $B$ given by the sequence of source $S_B=(1,5,3,4)$. }
    \label{fig}
\end{figure}
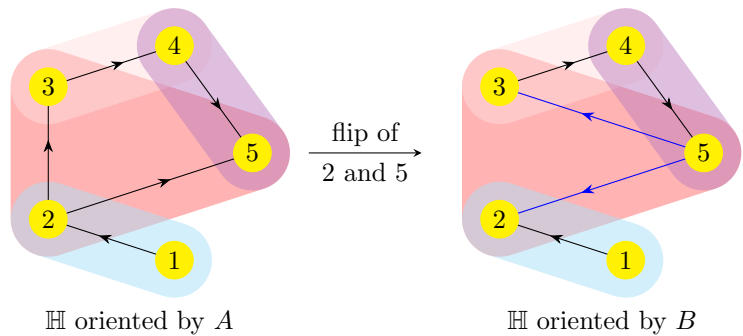
    
\end{example}

\begin{definition}
    Let $A$ and $B$ be two distinct orientations of $\HH$ related by a flip \flip and where $A$ is acyclic. This flip is a \defn{coherent flip} if $B$ is an acyclic orientation.
\end{definition}

That is to say that all cover relations in $P_\HH$ are given by coherent flips. We now have the following result.

\begin{lemma} \label{lemprecoherent}
    Let $A$ and $B$ be two distinct orientations of $\HH$ related by a flip \flip and where $A$ is acyclic. This flip is coherent if and only if it is pre-coherent.
\end{lemma}

\begin{proof}

    Let $A$ and $B$ be two distinct acyclic orientations such that \flip \hspace{-0.2cm}. Since $A\neq B$, there exists a hyperedge $H\in \HH$ such that $\{i,j\}\subseteq \HH$ and where $A(H)=i$. Thus, there is an edge $[i,j]$ in $A$. Suppose that this flip is not pre-coherent. This implies that there is a non-edge path $\lambda$ in $A$ starting from vertex $i$ to the vertex $j$. Two cases are possible. 

    In the first case, suppose there exists a hyperedge $H' \in \HH$ such that $\{i, j, k\} \subseteq H'$ where $k$ is a vertex that is part of the path $\lambda$. Since the subpath $[k,\dots,j]\subset \lambda$ is in $A$ and $i,j\in H'$, there is a path $\gamma=[i,k,\dots,j]$ in $A$. Moreover, since $A$ is acyclic, the entries of $\gamma$ are distinct. Since $\{k,j\}\subset H'$, this implies that after the flip, we have the existence of the edge $[j, k]$ in $B$. The existence of the edge $[a,b]\subset \gamma$ gives us the existence of hyperedge $E\in \HH$ such that $A(E)=a\neq i$ in $B$. This gives us the existence of the edge $[a,b]$ in $B$ and thus, the subpath $\alpha=[k,\dots,j]$ of $\gamma$ in $A$ remains in $B$. Therefore, we obtain a cycle from $j$ to $j$ in $B$ by concatenating the edge $[j, k]$ with the path $\alpha$. (See Figure~\ref{fig} for an example of the first case).

    In the second case, there is no hyperedge $H' \in \HH$ with $A(H')=i$ such that $\{i, j, k\} \subseteq H'$ where $k$ is a vertex that is part of the path $\lambda$. This implies that after the flip, the path $\lambda$ stay unchanged and thus remains in $B$. Moreover, since $[i,j]$ is an edge in $A$, there exists a hyperedge $H'\in \HH$ such that $i,j\in H'$. Hence, this gives the existence of the edge $[j,i]$ is in $B$. Therefore, we obtain a cycle from $i$ to $i$ in $B$ by concatenating the path $\lambda$ with the edge $[j,i]$. This proves the forward direction of the lemma. 

    Suppose that \flip is pre-coherent but not coherent, this implies that there exists a cycle in $B$. Since $A$ is acyclic and that $A$ and $B$ are related by the flip \flip \hspace{-0.2cm}, then either the edge $[j,i]$ is part of a cycle from vertex $i$ to $i$ in $B$ or there exists an edge $[j, k]$, where $k$ is in a hyperedge $H' \in \HH$ such that $\{i, j, k\} \subseteq H'$, that is part of a cycle from $j$ back to $j$ in $B$. In the first scenario, since the edge $[j,i]$ is in $B$, by definition of a flip, the edge $[i,j]$ can not be in $B$. This implies that there is a path that is not an edge from vertex $i$ to vertex $j$ in $B$. Therefore there is a non-edge path from vertex $i$ to vertex $j$ in $A$. In the second scenario, there exists a hyperedge $H'\in \HH$ such that $\{i,j,k\}\subseteq H'$. Thus, this gives us the existence of the edge $[i,k]$ in $A$ and the existence of the path $\alpha=[k,\dots,j]$ (subpath of the cycle in $B$) in $A$. Therefore, by concatenating the edge $[i,k]$ with the path $\alpha$ we obtain a non-edge path from vertex $i$ to vertex $j$ in $A$. Hence, in both cases, we reach a contradiction, showing that this flip cannot be a pre-coherent flip, thereby establishing the acyclicity of $B$. This proves the reverse implication of the lemma and concludes the proof.
\end{proof}

 \vspace{-0.4cm}

\section{A source characterization of $P_\HH$}

\subsection{Source sequences, small and minimized hyperedges}

Given two distinct acyclic orientations $A$ and $B$ of a hypergraph $\HH$, the sufficient condition for the backward direction of Theorem~\ref{thm} is given by:

$$ A(H)\leq B(H) \textit{ for all } H \in \HH.$$

To simplify this condition, we introduce the concept of source sequences and the notion of comparability of these sequences.

\begin{definition}
    Let $A$ be an orientation on $\HH$. Let $S_A=\left(A(H)\right)_{H\in \HH}$ be defined as the \defn{source sequence} of the orientation $A$.
\end{definition}

\pagebreak

\begin{definition}
    Let $A$ and $B$ be two distinct orientations on $\HH$. Two source sequences are said to be \defn{comparable} if $A(H)\leq B(H)$ for all $H\in \HH$. This is denoted as $S_A\leq S_B$
\end{definition}

Let us now introduce two related notions of hyperedges in a given hypergraph $\HH$. Those notions will be crucial to prove that after a given flip, two source sequences will remain comparable.

\begin{definition} \label{small}
    Let $A$ and $B$ be distinct acyclic orientations of $\HH$ such that $S_A < S_B$. Let $S_{i,h}$ be the set of all hyperedges $H\in \HH$ satisfying the following conditions:
    
    \begin{itemize}
        \item [(1)] $A(H)=i$,
        \item [(2)] $A(H)< B(H)$, and
        \item [(3)] $h\in H$.
    \end{itemize} 
    
    When $S_{i,h}\neq \emptyset$, we say that a hyperedge $\mathbf{H}_{i,h}\in S_{i,h}$ is a \defn{small hyperedge} of $S_{i,h}$ if it satisfies the following condition:

    $$B(\mathbf{H}_{i,h})=\min\left\{B(H) : H \in S_{i,h}\right\}.$$

    When $h$ is omitted as a subscript in the notation, we consider the small hyperedge $\mathbf{H}_i$ of $S_i$ where the last condition $(3)$ of the hyperedges in $S_i$ to contain the vertex $h$ is removed.
\end{definition}

\begin{definition}\label{mini}
     Let $A$ and $B$ be distinct acyclic orientations of the hypergraph $\HH$ such that $S_A < S_B$. For given vertices $i$ and $h$ in $[n]$ such that $S_{i,h}\neq \emptyset$, we define the hyperedges $\mathbf{H}^{(k)}_{i,h}$ by the inductive procedure:

     \begin{align*}
         &\mathbf{H}^{(1)}_{i,h} = \mathbf{H}_{i,h_0}, \text{ where } h_0 = h, \\
         &\mathbf{H}^{(2)}_{i,h} = \mathbf{H}_{i,h_1}, \text{ where } h_1 = B(\mathbf{H}_{i,h_0}), \\
         &\text{and for }k \geq 3,   \\
         &\mathbf{H}^{(k)}_{i,h} = \mathbf{H}_{i,h_{k-1}},  \text{ where } h_{k-1} = B(\mathbf{H}_{i,h_{k-2}}).
     \end{align*} 
     
     Where $\mathbf{H}_{i,h_j}$ are small hyperedges of $S_{i,h_j}$. Note that $h_0\geq h_i \geq \dots$. Hence, this sequence of inequalities stabilizes. Consequently, there exists an integer $m$ such that $B(\mathbf{H}^{(m)}_{i,h}) = B(\mathbf{H}^{(m+s)}_{i,h})$ for any $s\in \NN$, where $\mathbf{H}^{(m)}_{i,h}$ and $\mathbf{H}^{(m+s)}_{i,h}$ are respectively small hyperedge of the sets $S_{i,{h_{m-1}}}$ and $S_{i,{h_{m+s-1}}}$. We define the integer $M$ as follows:

     $$M=\min\{m\in \NN | B(\mathbf{H}^{(m)}_{i,h}) = B(\mathbf{H}^{(m+s)}_{i,h}) \textit{ for all } s\in \NN \}.$$

    We say that the hyperedge $\mathbf{H}^{(M)}_{i,h}$ is a \defn{minimized hyperedge} associated to the tuple $(i,h)$. Note that a minimized hyperedge may not be unique.

    \end{definition}

\pagebreak

\begin{example}

    Figure \ref{fig7} illustrates a minimized hyperedge $H^{(M)}_{1,7}$ on the hypergraph $\HH = \{H_1 =157, H_2 =147, H_3 =17, H_4 = 1236, H_5 = 134\}$ oriented by $A = \mathcal{O}_{1234567}$ according to the orientation $B = \mathcal{O}_{5234176}$. Their source sequences are respectively $S_A=(1,1,1,1,1)$ and $S_B=(5,4,7,2,3)$. According to the inductive procedure given in Definition~\ref{mini}, we can compute the following: 

    \begin{align*}
     &\mathbf{H}^{(1)}_{1,7} = \mathbf{H}_{1,7} = H_2 \quad \text{where } h_0 = 7,\\
     &\mathbf{H}^{(2)}_{1,7} = \mathbf{H}_{1,4}=H_5 \quad \text{where } h_1 = B(\mathbf{H}_{1,7})=4,\\
     &\mathbf{H}^{(3)}_{1,7} = \mathbf{H}_{1,3}=H_4 \quad \text{where } h_2 = B(\mathbf{H}_{1,4})=3,\\
     &\mathbf{H}^{(4)}_{1,7} = \mathbf{H}_{1,2}=H_4 \quad \text{where } h_3 = B(\mathbf{H}_{1,3})=2.\\
     \end{align*}

    We have that $B(\mathbf{H}^{(3)}_{1,7})=B(\mathbf{H}^{(4)}_{1,7})$. Hence, the following:
    
    $$3=\min\{m\in \NN | B(\mathbf{H}^{(m)}_{1,h}) = B(\mathbf{H}^{(m+s)}_{1,h}) \textit{ for all } s\in \NN \}.$$
    
    Thus, a possible minimized hyperedge $\mathbf{H}^{(3)}_{1,7}$ is the hyperedge $H_4$.
    
    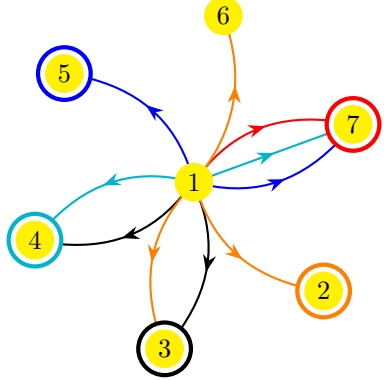
\begin{figure}[H]
    \centering
   \begin{tikzpicture}[
    mid arrow/.style={postaction={decorate,decoration={markings, mark=at position .6 with {\arrow{Stealth}}
    }}},]

    \draw[color=blue,mid arrow,thick] (0,0) to[bend right] ({1.95*cos(15)},{1.95*sin(15)});
    \draw[color=red,mid arrow,thick] (0,0) to[bend left] ({1.95*cos(25)},{1.95*sin(25)});
    \draw[color=Turquoise,mid arrow,thick] (0,0) to ({1.9*cos(20)},{1.9*sin(20)});

    \draw[color=orange,mid arrow,thick] (0,0) to[bend right] ({2.25*cos(80)},{2.25*sin(80)});

    \draw[color=blue,mid arrow,thick] (0,0) to[bend right] ({1.95*cos(135)},{1.95*sin(135)});

    \draw[color=Turquoise,mid arrow,thick] (0,0) to[bend right] ({1.95*cos(195)},{1.95*sin(195)});
    \draw[color=black,mid arrow,thick] (-0.02,0) to[bend left] ({1.95*cos(205)},{1.95*sin(205)});

    \draw[color=black,mid arrow,thick] (-0.02,0) to[bend left] ({1.95*cos(265)},{1.95*sin(265)});
    \draw[color=orange,mid arrow,thick] (0.02,0) to[bend right] ({1.95*cos(255)},{1.95*sin(255)});

    \draw[color=orange,mid arrow,thick] (0.01,0) to[bend right] ({1.95*cos(315)},{1.95*sin(315)});


    \node[draw=yellow, fill=yellow, ellipse, minimum width=0.5cm, minimum height=0.5cm, inner sep=5pt] at (0,0) {};
    
    \node[draw=yellow, fill=yellow, ellipse, minimum width=0.5cm, minimum height=0.5cm, inner sep=5pt] at ({2.25*cos(20)},{2.25*sin(20)}) {};
    \node[draw=yellow, fill=yellow, ellipse, minimum width=0.5cm, minimum height=0.5cm, inner sep=5pt] at ({2.25*cos(80)},{2.25*sin(80)}) {};
    \node[draw=yellow, fill=yellow, ellipse, minimum width=0.5cm, minimum height=0.5cm, inner sep=5pt] at ({2.25*cos(140)},{2.25*sin(140)}) {};
    \node[draw=yellow, fill=yellow, ellipse, minimum width=0.5cm, minimum height=0.5cm, inner sep=5pt] at ({2.25*cos(200)},{2.25*sin(200)}) {};
    \node[draw=yellow, fill=yellow, ellipse, minimum width=0.5cm, minimum height=0.5cm, inner sep=5pt] at ({2.25*cos(260)},{2.25*sin(260)}) {};
    \node[draw=yellow, fill=yellow, ellipse, minimum width=0.5cm, minimum height=0.5cm, inner sep=5pt] at ({2.25*cos(320)},{2.25*sin(320)}) {};

    \node[draw=red, ellipse, minimum width=0.7cm, minimum height=0.7cm, inner sep=5pt, ultra thick] at  ({2.25*cos(20)},{2.25*sin(20)}) {};

    \node[draw=blue, ellipse, minimum width=0.7cm, minimum height=0.7cm, inner sep=5pt, ultra thick] at  ({2.25*cos(140)},{2.25*sin(140)}) {};

    \node[draw=Turquoise, ellipse, minimum width=0.7cm, minimum height=0.7cm, inner sep=5pt, ultra thick] at  ({2.25*cos(200)},{2.25*sin(200)}) {};

    \node[draw=black, ellipse, minimum width=0.7cm, minimum height=0.7cm, inner sep=5pt, ultra thick] at  ({2.25*cos(260)},{2.25*sin(260)}) {};

    \node[draw=orange, ellipse, minimum width=0.7cm, minimum height=0.7cm, inner sep=5pt, ultra thick] at  ({2.25*cos(320)},{2.25*sin(320)}) {};

    \node (7) at (0,0) {$1$};

    \node (6) at ({2.25*cos(20)},{2.25*sin(20)}) {$7$};
    \node (5) at ({2.25*cos(80)},{2.25*sin(80)}) {$6$};
    \node (4) at ({2.25*cos(140)},{2.25*sin(140)}) {$5$};
    \node (3) at ({2.25*cos(200)},{2.25*sin(200)}) {$4$};
    \node (2) at ({2.25*cos(260)},{2.25*sin(260)}) {$3$};
    \node (1) at ({2.25*cos(320)},{2.25*sin(320)}) {$2$};
    
    \end{tikzpicture}
    \caption{The hypergraph $\HH=\{\textcolor{blue}{157},\textcolor{Turquoise}{147},\textcolor{red}{17},\textcolor{orange}{1236},134\}$ oriented by $A=\mathcal{O}_{1234567}$. The colored circle around some vertices represent the sources of each hyperedges when $\HH$ is oriented by $B=\mathcal{O}_{5234176}$.}
    \label{fig7}
\end{figure}
\end{example}

\subsection{Non-edge paths}

Let $A$ and $B$ be two distinct acyclic orientations of a hypergraph $\HH$, such that $S_A < S_B$. To prove the backward direction of Theorem~\ref{thm}, we need to show the existence of a sequence of some $k$ increasing flips

\begin{equation*}
    \flippp
\end{equation*} 

such that

$$S_A < S_{A_1} < S_{A_2} < \cdots < S_{A_k} <  S_B.$$

We claim that there exists $i<j$ in $[n]$ such that \flipppp where $S_A<S_{A'}\leq S_B$. Once we show this claim, the main theorem will follow by induction on the distance from $A$ to $B$ defined by $\sum_{H} (B(H)-A(H))$.

\pagebreak

Hence, given two distinct acyclic orientations $A$ and $B$, we consider a vertex $i$ for which there exists a hyperedge $H\in\HH$ such that $A(H)=i\neq B(H)$. This implies $S_i\neq \emptyset$. Thus, we can consider a small hyperedge $\mathbf{H}_i$ of the set $S_i$ of $\HH$. We look at the flip \flipppp where $i=A(\mathbf{H}_i)$ and $j=B(\mathbf{H}_i)$. 

\begin{proposition} \label{showsmall}
Consider the flip \flipppp where $i = A(\mathbf{H}_i)$ and $j = B(\mathbf{H}_i)$ and where $\mathbf{H}_i$ is a small hyperedge of the set $S_i$. Then $S_{A'} \leq S_B$.
\end{proposition}

\begin{proof}

Suppose that $S_{A'}$ and $S_B$ are not comparable. This implies that there exists a hyperedge $H\in \HH$ such that $A'(H)>B(H)$. Since $A(H)\neq B(H)$, this implies that $A(H)=i$ and that $A'(H)=B(\mathbf{H}_i)$. Therefore, we have $B(H) < B(\mathbf{H}_i)$ and $i\in H$, which contradicts the Definition~\ref{small} of a small hyperedge $\mathbf{H}_i$ of the set $S_i$.

\end{proof}

Hence, if the flip \flipppp is coherent we get a first increasing flip from $A$ to $B$ in $P_\HH$. In the case where $A'$ is not acyclic, we must establish some intermediate results. Since $A$ is acyclic while $A'$ is not, Lemma~\ref{lemprecoherent} implies that the flip $\flipppp$ is not pre-coherent, leading to the following key result.

\begin{proposition} \label{mustmove}
Let $A$ and $B$ be two distinct acyclic orientations of $\HH$. For a hyperedge $H \in \HH$ with $i=A(H)< j=B(H)$, if the flip $\flipppp$ is not pre-coherent, then for any non-edge path $\lambda = [i,h, \dots, j]$ in $A$, there exists a hyperedge $H' \in \HH$ such that:
\begin{itemize}
    \item [(1)] $A(H') \neq B(H')$, and
    \item [(2)] $[A(H'), h'] \subset \lambda$, where $h' \in H'$.
\end{itemize}
Moreover, in the case where $\{i,j,h\}\subseteq F'$, for some hyperedge $F'\in \HH$ with $A(F')=i$ and $B(F')=j$, we can choose $H'$ such that $A(H')\neq i$.

\end{proposition}

\begin{proof}
 
    Consider a hyperedge $H\in \HH$ such that $A(H)\neq B(H)$ and denote $A(H)=i$ and $B(H)=j$. Consider the flip \flipppp is not pre-coherent. 
    
    Given $H\in \HH$ and the flip \flipppp, we will prove the existence of a hyperedge $H'\in \HH$ as mentioned in the proposition. Suppose, for the sake of contradiction, that there exists a non-edge path $\lambda = [i, h, \dots, j]$ in $A$ and no hyperedge $H' \in \HH$ satisfying conditions $(1)$ and $(2)$. This assumption implies that every hyperedge $E \in \HH$ containing the vertices $a$ and $b$, where $[a, b] \subset \lambda$, and satisfying $A(E)=a$, must also satisfy $A(E) = B(E)$. If not, there would exist a hyperedge $G \in \HH$ such that $A(G) = a \neq B(G)$ that contains $b$ and such that $[A(G), b] \subset \lambda$, contradicting our hypothesis on $\lambda$.
    
    Now, since $\lambda$ is in $A$, for every edge $[a,b] \subset \lambda$ there is a hyperedge $F\in \HH$ such that $A(F)=a$ and $b\in F$. Therefore, we have that $A(F)=B(F)=a$ and $F$ gives us the existence of $[a,b]$ in $B$. Hence, $\lambda$ is in $B$.
    
    Now, since $A(H)=i\neq B(H)=j$, the edge $[j,i]$ is in $B$. This gives us the existence of a cycle from $j$ to $j$ in $B$ formed by the concatenation of $\lambda$ with the edge $[j,i]$ in $B$. This is a contradiction since $B$ is acyclic. Therefore, a hyperedge $H'$ satisfying $(1)$ and $(2)$ must exist. 
    
    Moreover, suppose that the only hyperedges $G'\in \HH$ satisfying the conditions $(1)$ and $(2)$ are such that $A(G')=i$. That is, every hyperedge $E \in \HH$ containing the vertices $a$ and $b$, where $[a, b] \subset [h,\dots,j]$, and satisfying $A(E)=a$, must also satisfy $A(E) = B(E)$.
    
    Using the same arguments from the second and third paragraph of the proof we have $[h,\dots,j]$ is in $B$. Suppose we have $F'$ with $\{i,j,h\}\subseteq F'$ satisfying $(1)$ and $(2)$ such that $B(F')=j$, then $[j,h]$ exists in $B$. This gives us the existence of a cycle from $j$ to $j$ in $B$. This cycle is formed by concatenating the edge $[j,h]$ with the path $\alpha=[h,\dots,j]$. This is a contradiction since $B$ is acyclic. Therefore, our hypothesis fails and there must be a $H'$ with $A(H')\neq i$ satisfying $(1)$ and $(2)$.

\end{proof}

To clarify the notation in the proof, we introduce the following concept.

\begin{definition}\label{simple}

Let $A$ and $B$ be distinct acyclic orientations of $\HH$. For a hyperedge $H \in \HH$, where $i = A(H) < j = B(H)$, let $\lambda = [i, h, \dots, j]$ be a non-edge path from $i$ to $j$ in $A$. We say that the tuple $(i, j)_\lambda$ is \defn{simple} in $A$ if for all hyperedges $H'\in \HH$ satisfying the conditions $(1)$ and $(2)$ of Proposition~\ref{mustmove}, we have that $A(H')=i$ and $B(H')\neq j$.

\end{definition}

\begin{example}
    Let $\HH$ be the following hypergraph
    
    $$\begin{matrix}
        \HH= & \{H_1=12, H_2=13, H_3=23, H_4=14, H_5=45, \\
        & H_6=25, H_7=16, H_8=67, H_9=78, H_{10}=28\}
    \end{matrix}.$$ 
    
    Let us consider the two distinct acyclic orientations $A=\mathcal{O}_{13465782}$ and $B=\mathcal{O}_{34567821}$. Their source sequences are respectively $S_A=(1,1,3,1,4,5,1,6,7,8)$ and $S_B=(2,3,3,4,4,5,6,6,7,8)$. In $A$, let us consider all the paths from the vertex $1$ to the vertex $2$.
    
    $$\begin{matrix}
        \lambda_1=[1,2] && \lambda_2=[1,3,2] && \lambda_3=[1,4,5,2] && \lambda_4=[1,6,7,8]
    \end{matrix}.$$
    
    Figure \ref{fig8} illustrates tuples $(1,2)_{\lambda_2}, (1,2)_{\lambda_3}, (1,2)_{\lambda_4}$ that are simple in $A$. In this example, for all the hyperedges $H\in \HH$ such that $A(H)\neq B(H)$, we have $A(H) = 1$.
     
     $$\begin{matrix}
        A(H_1)=1\neq B(H_1)=2 && A(H_2)=1\neq B(H_2)=3 \\ A(H_4)=1 \neq B(H_4)=4 && A(H_7)=1 \neq B(H_7)=6
    \end{matrix}.$$

    Moreover, for hyperedges $H_2,H_4,H_7$, we have: 
    
    $$\begin{matrix}
        [A(H_2),B(H_2)]\subset \lambda_2 && [A(H_4), B(H_4)]\subset \lambda_3 && [A(H_7),B(H_7)]\subset \lambda_4
    \end{matrix}.$$
    
    The above gives us that the hyperedges $H_2,H_4,H_7\in \HH$ are satisfying condition $(1)$ and $(2)$ from Proposition~\ref{mustmove}, respectively for the path $\lambda_2,\lambda_3,\lambda_4$. Moreover, since $B(H_2), B(H_4),$ $B(H_7)\neq 2$, we can conclude that the tuple $(1,2)_\lambda$ is simple in $A$ for the path $\lambda=\lambda_2,\lambda_3,\lambda_4$ from vertex $i$ to $j$ in $A$.

    \begin{figure}[H]
    \centering
   \begin{tikzpicture}[scale=1,
    mid arrow/.style={postaction={decorate,decoration={markings, mark=at position .6 with {\arrow{Stealth}}
    }}},]


    \draw[color=red,thick,mid arrow] (0,0) -- (8,0);

    \draw[color=cyan,thick,mid arrow] (0,0) -- (4,0.75);
    \draw[color=pink,thick,mid arrow] (4,0.75) -- (8,0);

    \draw[color=violet,thick,mid arrow] (0,0) -- (2,1.5);
    \draw[color=orange,thick,mid arrow] (2,1.5) -- (6,1.5);
    \draw[color=Turquoise,thick,mid arrow] (6,1.5) -- (8,0);

    \draw[color=green,thick,mid arrow] (0,0) -- (1,2.25);
    \draw[color=black,thick,mid arrow] (1,2.25) -- (4,3);
    \draw[color=purple,thick,mid arrow] (4,3) -- (7,2.25);
    \draw[color=blue,thick,mid arrow] (7,2.25) -- (8,0);

    \node[draw=white,fill=white, ellipse, minimum width=0.5cm, minimum height=0.5cm, inner sep=5pt] at (0,0) {};
    \node[draw=red,ultra thick,fill=white, ellipse, minimum width=0.5cm, minimum height=0.5cm, inner sep=5pt] at (8,0) {};

    \node[draw=cyan,ultra thick,fill=white, ellipse, minimum width=0.5cm, minimum height=0.5cm, inner sep=5pt] at (4,0.75) {};

    \node[draw=violet,ultra thick,fill=white, ellipse, minimum width=0.5cm, minimum height=0.5cm, inner sep=5pt] at (2,1.5) {};
    \node[draw=white,fill=white, ellipse, minimum width=0.5cm, minimum height=0.5cm, inner sep=5pt] at (6,1.5) {};

    \node[draw=green,ultra thick,fill=white, ellipse, minimum width=0.5cm, minimum height=0.5cm, inner sep=5pt] at (1,2.25) {};
    \node[draw=white,fill=white, ellipse, minimum width=0.5cm, minimum height=0.5cm, inner sep=5pt] at (4,3) {};
    \node[draw=white,fill=white, ellipse, minimum width=0.5cm, minimum height=0.5cm, inner sep=5pt] at (7,2.25) {};

    \node (1) at (0,0) {$1$};
    \node (2) at (8,0) {$2$};
    \node (3) at (4,0.75) {$3$};
    \node (4) at (2,1.5) {$4$};
    \node (5) at (6,1.5) {$5$};
    \node (6) at (1,2.25) {$6$};
    \node (7) at (4,3) {$7$};
    \node (8) at (7,2.25) {$8$};

    \end{tikzpicture}
    \caption{The hypergraph $\HH=\{\textcolor{red}{12},\textcolor{Cyan}{13},\textcolor{pink}{23},\textcolor{violet}{14},\textcolor{orange}{45},\textcolor{Turquoise}{25},\textcolor{green}{16},67,\textcolor{purple}{78},\textcolor{blue}{28}\}$ oriented by $A=\mathcal{O}_{13465782}$. The colored circle around some vertices represent the sources of each hyperedges when $\HH$ is oriented by $B=\mathcal{O}_{34567821}$.}
    \label{fig8}
\end{figure}
\end{example}

\pagebreak

\subsection{Source path $\kappa$}

To show the existence of coherent flip \flipppp such that:

\begin{itemize}
    \item[(1.)] $\sum_{H\in \HH} (B(H)-A'(H))< \sum_{H\in \HH} (B(H)-A(H))$, and
    \item[(2.)] $S_A < S_{A'}\leq S_B$,
\end{itemize}

we will construct a useful path in $A$ denoted $\kappa$ and referred as a \defn{source path} of a non-coherent flip \flipppp \hspace{-0.2cm}. In the setup of a source path, we include the notion of paths of length $0$ in A. More specifically, this corresponds to a sequence consisting of a single vertex. Below, we provide an algorithm to construct $\kappa$. 

In this algorithm, we denote $[\hat{\imath}, \dots, j] = [h, \dots, j]$ for a path $[i, h, \dots, j]$.

\begin{algorithm}[H] 
\caption{Construction of the Path $\kappa$}
\label{alg:kappa6}

\textbf{Input:}
\justifying{Two distinct acyclic orientations $A$ and $B$ on $\HH$; A non-coherent flip \hspace*{13.5mm} \flipppp \hspace{-0.2cm}, where $i=A(H)< j=B(H)$ for some $H \in \HH$.}

\textbf{Output:} Constructed path $\kappa$.

    \begin{algorithmic}[1] 
    \State Consider $\mathbf{H}_{i}$, a small hyperedge for the set $S_{i}$;
    \State  Set $j \leftarrow B(\mathbf{H}_{i})$;
    \State  Set $\kappa = [i]$, where $i$ is the initial vertex.
    \While{a non-edge path exists from $i$ to $j$}
        \While{there is a non-edge path $\lambda$ in $A$ from $i$ to $j$ such that $(i,j)_\lambda$ is not simple in $A$}
        \State  \justifying{Locate the first vertex $i'$ in $\lambda$ starting from $i$ such that $i'\neq i$ is the source of a \hspace*{9.5mm} hyperedge $H' \in \HH$ with $A(H') \neq B(H')$ and where $h'\in H'$ for $[i',h']\subset \lambda$.} 
        \State Append $[\hat{\imath},\dots,i']\subset \lambda$ to $\kappa$;
        \State Consider $\mathbf{H}_{i'}$, a small hyperedge of the set $S_{i'}$;
        \State  Set $i \leftarrow i'$ and $j \leftarrow B(\mathbf{H}_{i'})$.
            \EndWhile
            \State Let $\lambda=[i,h,\dots,j]$ be a non edge path in $A$ such that $(i,j)_\lambda$ is simple in $A$;
            \State Consider a minimized hyperedge $\mathbf{H}^{(M)}_{i,h}$;
            \State  Set $j \leftarrow B(\mathbf{H}^{(M)}_{i,h})$.

    \EndWhile

    \State  \textbf{Return} the path $\kappa$.
    \end{algorithmic}
\end{algorithm} 

\begin{example} \label{source path}

    Let $\HH=\{H_1=123,H_2=34,H_3=245,H_4=46,H_5=67,H_6=57\}$ and let us consider the orientations $A=\mathcal{O}_{1346725}$ and $B=\mathcal{O}_{7523416}$. Their source sequences are respectively $S_A=(1,3,4,4,6,7)$ and $S_B=(2,3,5,4,7,7)$.
    
    Let $i=1$ and $j=2$. Since there is a non-edge path $\lambda_1=[1,3,4,2]$ from vertex $i$ to vertex $j$ in $A$, Lemma~\ref{lemprecoherent} tells us that the flip \flipppp is not coherent. We then construct the source path associated to the non-coherent flip \flipppp \hspace{-0.2cm}. Hence, the source path will start with $1$.
    
    Since $H_1$ is the only hyperedge containing the vertex $i=1$, we let $j=B(H_1)=2$ and consider the only non-edge path $\lambda_1$. The first source associated to a vertex $i'$ in $\lambda_1$ such that $A(H)=i'$ and $B(H)\neq i'$ is given by $A(H_3)=i'=4$. Since $2\in H_3$ we append $[3,4]$ to the source path.
    
    Let us start the process again by considering $i=4$. Since $H_4$ is a small hyperedge with source $A(H_4)=4$, we set $j=B(H_4)=6$. Let us consider the only non-edge path $\lambda_2=[4,6,7,5]$ from vertex $i$ to vertex $j$.  The first source associated to a vertex $i'$ in $\lambda_2$ such that $A(H)=i'$ and $B(H)\neq i'$ is given by $A(H_5)=i'=6$. Since $7\in H_5$ we append $[6]$ to the source path.
    
    Let us start the process again by considering $i=6$. Since $H_6$ is a small hyperedge with source $A(H_6)=6$, we set $j=B(H_6)=7$. Since there is no non-edge path from vertex $6$ to vertex $7$ and no other hyperedge containing the vertices $\{6,7\}$, the process stop. 
    
    Therefore, the source path $\kappa$ associated to the non-coherent flip \flipppp is given by $\kappa=[1,3,4,6]$. This source path $\kappa$ related to this example is shown in Figure~\ref{fig2}.
    
    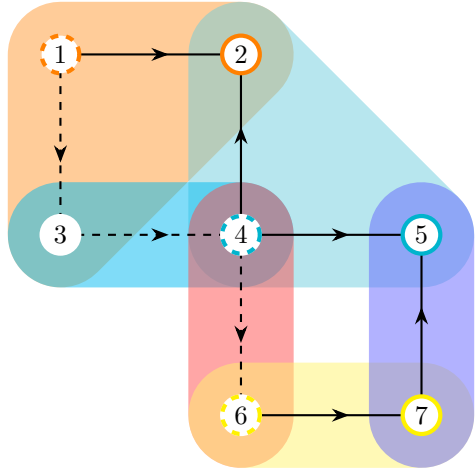
\begin{figure}[H]
    \centering
   \begin{tikzpicture}[scale=1.2,
    mid arrow/.style={postaction={decorate,decoration={markings, mark=at position .6 with {\arrow{Stealth}}
    }}},]


    \draw[dashed,thick,mid arrow] (0,0) -- (2,0);
    \draw[dashed,thick,mid arrow] (0,2) -- (0,0);
    \draw[thick,mid arrow] (0,2) -- (2,2);
    \draw[thick,mid arrow] (2,0) -- (2,2);
    \draw[dashed,thick,mid arrow] (2,0) -- (2,-2);
    \draw[thick,mid arrow] (2,0) -- (4,0);
    \draw[thick,mid arrow] (4,-2) -- (4,0);
    \draw[thick,mid arrow] (2,-2) -- (4,-2);

    \node[ fill=white, ellipse, minimum width=0.5cm, minimum height=0.5cm, inner sep=5pt] at (0,2) {};
    \node[draw=orange,dashed, ultra thick, ellipse, minimum width=0.5cm, minimum height=0.5cm, inner sep=5pt] at (0,2) {};
    \node[draw=orange,ultra thick,fill=white, ellipse, minimum width=0.5cm, minimum height=0.5cm, inner sep=5pt] at (2,2) {};
    \node[draw=white,ultra thick, fill=white, ellipse, minimum width=0.5cm, minimum height=0.5cm, inner sep=5pt] at (0,0) {};
    \node[draw=white,ultra thick, fill=white, ellipse, minimum width=0.5cm, minimum height=0.5cm, inner sep=5pt] at (2,0) {};
    \node[draw=Turquoise,ultra thick,dashed, ellipse, minimum width=0.5cm, minimum height=0.5cm, inner sep=5pt] at (2,0) {};
    \node[draw=white,ultra thick, fill=white, ellipse, minimum width=0.5cm, minimum height=0.5cm, inner sep=5pt] at (2,-2) {};
    \node[draw=yellow,ultra thick, dashed, ellipse, minimum width=0.5cm, minimum height=0.5cm, inner sep=5pt] at (2,-2) {};
    \node[draw=Turquoise,ultra thick, fill=white, ellipse, minimum width=0.5cm, minimum height=0.5cm, inner sep=5pt] at (4,0) {};
    \node[draw=yellow,ultra thick, fill=white, ellipse, minimum width=0.5cm, minimum height=0.5cm, inner sep=5pt] at (4,-2) {};

    \node (1) at (0,2) {$1$};
    \node (2) at (2,2) {$2$};
    \node (3) at (0,0) {$3$};
    \node (4) at (2,0) {$4$};
    \node (6) at (2,-2) {$6$};
    \node (5) at (4,0) {$5$};
    \node (7) at (4,-2) {$7$};

    \begin{pgfonlayer}{background}
        \draw[orange!40,fill=orange!40,line width=14mm,line cap=round,line join=round] (1.center) -- (2.center) -- (3.center) -- cycle;
        \draw[opacity=.5,Turquoise!60,line width=14mm,line cap=round,line join=round] (2.center) -- (4.center) -- (5.center) -- cycle;
        \draw[opacity=.5,cyan!80,fill=cyan!80,line width=14mm,line cap=round,line join=round] (3.center) -- (4.center) -- cycle;
        \draw[opacity=.5,red!70,fill=red!70,line width=14mm,line cap=round,line join=round] (4.center) -- (6.center) -- cycle;
        \draw[opacity=.5,yellow!70,fill=yellow!70,line width=14mm,line cap=round,line join=round] (6.center) -- (7.center) -- cycle;
        \draw[opacity=.5,blue!60,fill=blue!60,line width=14mm,line cap=round,line join=round] (5.center) -- (7.center) -- cycle;
    \end{pgfonlayer}

    \end{tikzpicture}
    \caption{The hypergraph $\HH=\{123,34,245,46,57,67\}$ oriented by $A=\mathcal{O}_{1342567}$. The full circles around some vertices represent the sources of each hyperedges when $\HH$ is oriented by $B=\mathcal{O}_{7523416}$. The dotted circles represented the element $i'$ at each step in the Algorithm~\ref{alg:kappa6}. The dotted path $[1,3,4,6]$ represent the source path $\kappa$ of Example~\ref{source path}.}
    \label{fig2}
\end{figure}
    
\end{example}

\begin{proposition}\label{proofalgo}

    Given two distinct acyclic orientations $A$ and $B$ on $\HH$ and a non-coherent flip \flipppp, where $i=A(H)< j=B(H)$ for some $H \in \HH$, Algorithm~\ref{alg:kappa6} will stop and produce a source path.

\end{proposition}

\begin{proof}

    Let us consider two distinct acyclic orientations $A$ and $B$ on $\HH$. Let us consider the Algorithm~\ref{alg:kappa6}.

    Each step of the algorithm is well defined. Indeed, let us study the non-trivial steps. Step $4$ is given by Definition~\ref{preco} of a pre-coherent path and by Proposition~\ref{lemprecoherent}. Step $5$ is given by Definition~\ref{simple} of a tuple that is simple. Step $6$ is given by Proposition~\ref{mustmove} and Definition~\ref{simple}. Step $8$ is given by Definition~\ref{small} of a small hyperedge. Step $12$ is given by Definition~\ref{mini}.
    
    Since $i$ corresponds to the source of some hyperedge $H\in \HH$ such that $A(H)=i\neq B(H)$ and $j=B(H)$, the input conditions of Algorithm~\ref{alg:kappa6} are satisfied and this certified that the set $S_i$ is not empty. Hence, Algorithm~\ref{alg:kappa6} will go through Step $1$, Step $2$ and Step $3$. Thus, $\kappa$ will be a path of length $\geq 0$ and therefore, Algorithm~\ref{alg:kappa6} will produce a source path.

    Now, given the hypothesis of the proposition, let us show that Algorithm~\ref{alg:kappa6} will stop. Let us consider the construction of the path $\kappa$ in $A$ given by the Algorithm~\ref{alg:kappa6}. Suppose, for the sake of contradiction, that the algorithm will never stop. Since $\HH$ is a hypergraph on a finite number of vertices and since $A$ is acyclic this implies that the source path $\kappa$ will be of finite length.
    
    Thus, from a certain iteration of Algorithm~\ref{alg:kappa6}, we do not append $\kappa$ for an infinite consecutive number of steps. Let us consider the $k^{th}$ iteration of the first \textit{``while''} loop in Algorithm~\ref{alg:kappa6}, corresponding to the last iteration associated to an increasing length iteration of $\kappa$. 
    
    Under the assumption that the algorithm does not terminate, there exist an infinite sequence of minimized hyperedges $(\mathbf{H}^{(M)}_{i,y_s})_{s\in \NN}$ such that for each of them, there exists at least one non-edge path from the last vertex $i$ to the vertex $B(\mathbf{H}^{(M)}_{i,y_s})$. Here, $s$ corresponds to the $s^{th}$ iteration of Step $12$ in Algorithm~\ref{alg:kappa6} and, the vertices $y_{s}$ are associated to the paths $[i,y_s,\dots,j]$ in $A$.

    To complete the proof, we will show the existence of some paths in $B$ that form a cycle, leading to a contradiction of the hypothesis of $B$ being acyclic.

    \pagebreak
    
    Since the algorithm does not stop after the $s^{th}$ iteration of Step $12$, we have a non-edge path from $i$ to $B(\mathbf{H}^{(M)}_{i,y_s})$ in $A$. The existence of the $(s+1)^{th}$ iteration of Step $12$ and the fact that the length of the path $\kappa$ stops increasing after the $k^{th}$ iterations, give us the existence of a path of the form $[i,y_{s+1},\dots,B(\mathbf{H}^{(M)}_{i,y_s})]$ in $A$ and the tuple $(i,B(\mathbf{H}^{(M)}_{i,y_s}))_\lambda$ is simple in $A$ for all path $\lambda$ from vertex $i$ to $B(\mathbf{H}^{(M)}_{i,y_s})$ in $A$. Hence, we have the existence of the following path in $B$:

    $$\alpha_{y_{s+1}}=\left[y_{s+1},\dots,B\left(\mathbf{H}^{(M)}_{i,y_s}\right)\right].$$
    
    Now, consider the recursive construction of a minimized hyperedge $\mathbf{H}^{(M)}_{i,y_s}$ given by Definition~\ref{mini}. For any vertex $y_s$, we have that the vertices $B(\mathbf{H}^{(t)}_{i,y_s})$ and $B(\mathbf{H}^{(t-1)}_{i,y_s})$ are contained in the hyperedge $ \mathbf{H}^{(t)}_{i,y_s}$. This gives us the existence of the following path in $B$:
    
    $$\alpha_{y_s}^t=\left[B\left(\mathbf{H}^{(t)}_{i,y_s}\right),B\left(\mathbf{H}^{(t-1)}_{i,y_s}\right)\right],$$
    
    For $2\leq t\leq M$. Moreover, since Definition~\ref{mini} gives us $\mathbf{H}^{(1)}_{i,y_s}=\mathbf{H}_{i,y_s}$, we have that the vertex $B(\mathbf{H}^{(1)}_{i,y_s})$ and the vertex $y_s$ are contained in $\mathbf{H}^{(1)}_{i,y_s}$. This gives us the existence of the following path in $B$:

    $$\alpha_{y_s}^1=\left[B\left(\mathbf{H}^{(1)}_{i,y_s}\right),y_s\right].$$

    Since we supposed that the sequence of minimized hyperedges $(\mathbf{H}^{(M)}_{i,y_s})_{s\in \NN}$ is infinite and since $\HH$ is finite, we have the existence of integers $u<v\in \NN$ greater than $k$, for which $\mathbf{H}^{(M)}_{i,y_u}=\mathbf{H}^{(M)}_{i,y_v}$. Therefore this gives us the existence of a cycle from the vertex $B(\mathbf{H}^{(M)}_{i,y_u})$ to the vertex $B(\mathbf{H}^{(M)}_{i,y_v})$ in $B$ given by the concatenation of the following paths in $B:$

    $$(\alpha_{y_{v}}^M \circ \cdots \circ \alpha_{y_{v}}^1 \circ \alpha_{y_{v}}) \circ (\alpha_{y_{v-1}}^M \circ \cdots \circ \alpha_{y_{v-1}}^1 \circ \alpha_{y_{v-1}}) \circ \cdots \circ (\alpha^{M}_{y_{u+1}} \circ \cdots \circ \alpha_{y_{u+1}}). $$

    This is a contradiction since $B$ is acyclic. Therefore, the sequence of minimized hyperedges $(\mathbf{H}^{(M)}_{i,y_s})_{s\in \NN}$ is finite. This implies that Algorithm \ref{alg:kappa6} terminates and this concludes the proof. 
    
\end{proof}

\pagebreak

\begin{example}
    Figure \ref{fig10} illustrates this following example. Let $\HH=\{H_1=12,H_2=145, H_3=134,H_4=127,H_5=25,H_6=37\}$ and let us consider the orientations $A=\mathcal{O}_{1673452}$ and $B=\mathcal{O}_{6734521}$. Their source sequence are respectively $S_A=(1,1,1,1,5,7)$ and $S_B=(6,4,3,2,5,7)$. 
    
    Let $i=1$ and let $j=2$. Since there is a non-edge path from vertex $i$ to vertex $j$, the flip \flipppp is non-coherent. Moreover, since, $H_1$ is a small hyperedge for the set $S_1$ and $B(H_1)=2=j$, let us consider Algorithm~\ref{alg:kappa6} for the vertices $i,j$. 
    
    From Definition~\ref{simple}, we have that $(i,j)_\lambda$ is simple for all path $\lambda$ from vertices $i$ to $j$ in $A$. Moreover, $\gamma=[1,5,2]$ is the only non-edge path from $i=1$ to $j=2$. Hence, according to the proof of Proposition~\ref{proofalgo}, $y_1=5$. Therefore, Step $12$ gives us:
    
    $$\begin{matrix}
        \mathbf{H}_{i,y_1}^{(1)}= \mathbf{H}_{i,y_1} &\longrightarrow & \mathbf{H}_{1,5}^{(1)}=H_2
    \end{matrix}$$

    and 

    $$\begin{matrix}
         \mathbf{H}_{i,y_1}^{(2)}=\mathbf{H}_{i,B(H_{i,y_1})} &\longrightarrow& \mathbf{H}_{1,5}^{(2)}=\mathbf{H}_{i,B(H_2)}= H_3
    \end{matrix}.$$

    We remark that $H_3$ is the only possible minimized hyperedge $\mathbf{H}^{(M)}_{i,y_1}$ and that the flip \linebreak \flipppsp is non-coherent. Therefore, we start the process again for $i=1$ and $j=B(H_3)=3$.  From Definition~\ref{simple}, we have that, $(i,j)_\lambda$ is simple for all path $\lambda$ from vertices $i$ to $j$ in $A$. Moreover, $\gamma'=[1,7,3]$ is the only non-edge path from $i=1$ to $j=B(H_3)=3$. Hence, according to the proof of Proposition~\ref{proofalgo}, $y_2=7$. Therefore, Step $12$ gives us:

    $$\begin{matrix}\mathbf{H}_{i,y_2}^{(1)}= \mathbf{H}_{1,7} & \longrightarrow & \mathbf{H}_{1,7}^{(1)}=H_4 \end{matrix}.$$

    From the above information, we can consider some paths in $B$ as described in the proof of Proposition~\ref{proofalgo}. Since $y_1=5$ and $B(\mathbf{H}_{i,y_1}^{(1)})=B(H_2)=4$ are in $H_2$ this gives us the existence of the following path in $B$:
    
    $$\begin{matrix} \alpha_{y_1}^1=[B(\mathbf{H}_{i,y_1}^{(1)}),y_i] & \rightarrow & \alpha_{5}^1=[B(\mathbf{H}_{1,5}^{(1)}),5] =[B(H_2),5] =[4,5].
    \end{matrix}$$

    Since $B(\mathbf{H}_{i,y_1}^{(1)})=B(H_2)=4$ and $B(\mathbf{H}_{i,y_1}^{(2)})=B(H_3)=3$ are in $H_3$ this gives us the existence of the following path in $B$:
    
    $$\begin{matrix} \alpha_{y_1}^2=[B(\mathbf{H}_{i,y_1}^{(2)}),B(\mathbf{H}_{i,y_1}^{(1)})] & \rightarrow & \alpha_{5}^2=[B(\mathbf{H}_{1,5}^{(2)}),B(\mathbf{H}_{1,5}^{(1)})] =[B(H_3),B(H_2)] =[3,4].
    \end{matrix}$$

    Since, the following tuple is simple
    
    $$(i,B(\mathbf{H}_{i,y_1}^{(2)}))_\gamma = (1,3)_{[1,7,3]},$$

    this gives us the existence of the following path in $B$:

    $$\begin{matrix} \alpha_{y_2}=[y_2,B(\mathbf{H}_{i,y_1}^{(2)})] & \rightarrow & \alpha_{7}=[7,B(H_3)] =[7,3].
    \end{matrix}$$

    Since $B(\mathbf{H}_{i,y_2}^{(1)})=B(H_4)=6$ and $y_2=7$ are in $H_4$ this gives us the existence of the following path in $B$:
    
    $$\begin{matrix} \alpha_{y_2}^1=[B(\mathbf{H}_{i,y_2}^{(1)}),y_2] & \rightarrow & \alpha_{7}^1=[B(\mathbf{H}_{1,7}^{(1)}),7] =[B(H_4),7] =[6,7].
    \end{matrix}$$
    
\end{example}

\input{exampleproof1}

\subsection{Source characterization of $P_\HH$}

Let us consider the hypergraph $\HH$ and two distinct acyclic orientations $A$ and $B$ such that $S_A< S_B$.

\begin{proposition}\label{mustcover}

    Let $A$ and $B$ be two distinct acyclic orientations of $\HH$ such that $S_A< S_B$. There always exist a pair of vertices $i< j$ in some hyperedge $H\in \HH$ such that 

    \begin{itemize}
        \item $A(H)=i$ and $B(H)=j$;
        \item The only paths between $i$ and $j$ are of length $1$
    \end{itemize}
     
\end{proposition}

\begin{proof}

    By construction of $\kappa$ given by Algorithm~\ref{alg:kappa6} we have that the last vertex $x$ of this path correspond to a source sequence of some small or minimized hyperedge $H\in \HH$ such that $A(H)=x$ and such that $A(H)\neq B(H)$. Also by construction of $\kappa$ we have that the only paths between $A(H)$ and $B(H)$ are edges. Therefore, we set this last vertex $A(H)=x$ of the source path $\kappa$ to correspond to $i$ and $B(H)$ to correspond to $j$. This proves the result.
\end{proof}

\begin{corollary}\label{mustflip}

    Let $A$ and $B$ be two distinct acyclic orientations of $\HH$ such that $S_A< S_B$. There always exists a coherent flip \flipppp such that $S_A < S_{A'} \leq S_B$.
    
\end{corollary}

\begin{proof}

    By construction of $\kappa$ given by Algorithm~\ref{alg:kappa6} we have that the last vertex $x$ of this path correspond to a source sequence of some small or minimized hyperedge $H\in \HH$ such that $A(H)=x$ and such that $A(H)\neq B(H)$. More specifically, $A(H)<B(H)$ since $S_A\leq S_B$.
    
    Also by construction of $\kappa$ we have that only the paths between $A(H)$ and $B(H)$ are edges. Therefore, we set the last vertex $A(H)=x$ of the source path $\kappa$ to correspond to $i$ and $B(H)$ to correspond to $j$. Proposition~\ref{mustcover} followed by Lemma~\ref{lemprecoherent} gives us that \flipss is coherent.
    
    Suppose that $H$ is a minimized hyperedge for a set $S_{i,h}$. By Definition~\ref{mini} we have that $H$ is a small hyperedge for a set $S_{i,h'}$. Since $H$ is a minimized hyperedge for a set $S_{i,h}$ we have that $B(H)= B(\mathbf{H}^{(M)}_{i,h'})$ but also that $B(H)=B(\mathbf{H}^{(M)}_{i,B(H)})$. This implies that for any hyperedge $H'\in\HH$ that contains vertices $i, B(H)$ such that $A(H')\neq B(H')$ with $A(H')=i$, we have that $B(\mathbf{H}^{(M)}_{i,B(H)})=B(H) \leq B(H')$. Therefore, for the flip \flipss \hspace{-0.2cm}, we get that  $S_A< S_{A'}\leq S_B$.
    
    In the case where $H$ is a small hyperedge for the set $S_i$, this implies that for any hyperedge $H'\in\HH$ that contains vertex $i$ and such that $A(H')\neq B(H')$ with $A(H')=i$, we have that $B(H) \leq B(H')$. Therefore, for the flip \flipss \hspace{-0.2cm}, we get that  $S_A< S_{A'}\leq S_B$.
\end{proof}

\begin{mythm}{A}
     For any acyclic orientations $A$ and $B$ of a hypergraph $\HH$ 

    $$A \leq B \Longleftrightarrow A(H)\leq B(H) \textit{ for all } H \in \HH.$$
\end{mythm}

\begin{proof}
    The forward direction is immediate as it holds for increasing flip by Definition~\ref{flip} given that any cover of $P_\HH$ is an increasing flip. For the backward direction, we have by Corollary~\ref{mustflip} that for any source sequence $S_A \leq S_B$ there exists an increasing coherent flip \flipppp such that $A'$ is acyclic, $A < A'$ and such that $S_A<S_{A'}\leq S_B$. Since the elements $i,j$ of the coherent flip \flipppp correspond to $A(H)=i$ and $B(H)=j$, we get that $A(H)\neq A'(H)=B(H)=j$ only for hyperedges $H\in \HH$ that contains $i,j$. Therefore, on the distance from $A$ to $B$ defined by $\sum_{H} (B(H)-A(H))$, we can conclude the result by induction.
\end{proof}

\section*{Acknowledgments}

The author would like to thank Nantel Bergeron, Vincent Pilaud and Mike Zabrocki.

\bibliographystyle{alpha}
\bibliography{bibliography}{}

\end{document}